\documentclass[onecolumn]{IEEEtran}

\usepackage{amssymb}
\usepackage{amsmath}
\usepackage{amsthm}
\usepackage{algorithmic}
\newtheorem{theorem}{Theorem}
\newtheorem{lemma}{Lemma}
\newtheorem{corollary}{Corollary}

\title{Step size adaptation in first-order method for stochastic strongly convex programming}

\author{
	Peng Cheng\\
	pc175@uow.edu.au
}

\begin{document}

\maketitle

\begin{abstract}

We propose a first-order method for stochastic strongly convex optimization that attains $O(1/n)$ rate of convergence, analysis show that the proposed method is simple, easily to implement, and in worst case, asymptotically four times faster than its peers. We derive this method from several intuitive observations that are generalized from existing first order optimization methods.

\end{abstract}

% investigating conditions of transplanting deterministic algorithms for stochastic programming, and by exploring the idea of updating lower bound functions adaptively to yield optimal worst-case bound of objective.

\section{Problem Setting}

In this article we seek a numerical algorithm that iteratively approximates the solution $w^*$ of the following strongly convex optimization problem:

\begin{equation}\label{eq1.0}
	w^*=\arg \min_{\Gamma_f} f(.)
\end{equation}

\noindent where $f(.): \Gamma_f \rightarrow \mathcal{R}$ is an unknown, not necessarily smooth, multivariate and $\lambda$-strongly convex function, with $\Gamma_f$ its convex definition domain. The algorithm is not allowed to accurately sample $f(.)$ by any means since $f(.)$ itself is unknown. Instead the algorithm can call stochastic oracles $\tilde{\omega}(.)$ at chosen points $\tilde{x}_1,\ldots,\tilde{x}_n$, which are unbiased and independent probabilistic estimators of the first-order local information of $f(.)$ in the vicinity of each $x_i$:

\begin{equation}\label{eq1.1}
	\tilde{\omega}(x_i) = \{\tilde{f}_i(x_i), \bigtriangledown \tilde{f}_i(x_i)\}
\end{equation}

\noindent where $\bigtriangledown$ denotes random subgradient operator, $\tilde{f}_i(.): \Gamma_f \rightarrow \mathcal{R}$ are independent and identically distributed (i.i.d.) functions that satisfy:

\begin{subequations}\label{eq1.2}
\begin{align}
	\text{(unbiased) } \mathbb{E}[\tilde{f}_i(.)] &= f(.) \quad \forall i \label{eq1.2a} \\
	\text{(i.i.d) }\mathrm{Cov} \left( \tilde{f}_i(.), \tilde{f}_j(.) \right) &= 0 \quad \forall i \neq j \label{eq1.2b}
\end{align}
\end{subequations}

Solvers to this kind of problem are highly demanded by scientists in large scale computational learning, in which the first-order stochastic oracle is the only measurable information of $f(.)$ that scale well with both dimensionality and scale of the learning problem. For example, a stochastic first-order oracle in structural risk minimization (a.k.a. training a support vector machine) can be readily obtained in $O(1)$ time \cite{bottou2008SGD}.

\section{Algorithm}

The proposed algorithm itself is quite simple but with a deep proof of convergence. The only improvement comparing to SGD is the selection of step size in each iteration, which however, results in substantial boost of performance, as will be shown in the next section.

\begin{figure*}
\noindent \rule{\linewidth}{0.5mm}\\
\textbf{Algorithm 1}\\
\rule{\linewidth}{0.3mm}
\begin{algorithmic}
\STATE Receive $x_1, \Gamma_f, \lambda$
\STATE $u_1 \gets 1$, $y_1 \gets x_1$
\STATE Receive $\tilde{f}_1(x_1), \bigtriangledown \tilde{f}_1(x_1)$
\STATE $\tilde{P}_1(.) \gets \Gamma_f \left\{ \tilde{f}_1(x_1)+\langle \bigtriangledown \tilde{f}_1(x_1), .-x_i \rangle+\frac{\lambda}{2} ||.-x_1||^2 \right\}$
\FOR {$i=2,\ldots,n$}
	\STATE $x_i \gets \arg\min \tilde{P}_{i-1}(.)$
	\STATE Receive $\tilde{f}_i(x_i), \bigtriangledown \tilde{f}_i(x_i)$
	\STATE $\tilde{p}_i(.) \gets \Gamma_f \left\{ \tilde{f}_i(x_i)+\langle \bigtriangledown \tilde{f}_i(x_i), .-x_i \rangle+\frac{\lambda}{2} ||.-x_i||^2 \right\}$
	\STATE $\tilde{P}_i(.) \gets (1-\frac{u_{i-1}}{2}) \tilde{P}_{i-1}(.) + \frac{u_{i-1}}{2} \tilde{p}_i(.)$
	\STATE $y_i \gets (1-\frac{u_{i-1}}{2}) y_{i-1} + \frac{u_{i-1}}{2} x_{i}$
	\STATE $u_i \gets u_{i-1}-\frac{u_{i-1}^2}{4}$
\ENDFOR
\STATE Output $y_n$
\end{algorithmic}
\rule{\linewidth}{0.5mm}
\end{figure*}

\section{Analysis}

The proposed algorithm is designed to generate an output $y$ that reduces the suboptimality:

\begin{equation}\label{eq3.1}
	S(y)=f(y)-\min f(.)
\end{equation}

\noindent as fast as possible after a number of operations. We derive the algorithm by several intuitive observations that are generalized from existing first order methods. First, we start from worst-case upper-bounds of $S(y)$ in deterministic programming:

\begin{lemma}\label{state3.1}

(Cutting-plane bound \cite{joachims2006cuttingplane}): Given $n$ deterministic oracles $\Omega_n=\{\omega(x_1),\ldots,\omega(x_n)\}$ defined by:

\begin{equation}\label{eq3.2}
	\omega(x_i)=\{f(x_i), \bigtriangledown f(x_i)\}
\end{equation}

If $f(.)$ is a $\lambda$-strongly convex function, then $\min f(.)$ is unimprovably lower bounded by:

\begin{equation}\label{eq3.1.5}
	\min f(.) \geq \max_{i=1 \ldots n} p_i(w^*) \geq \min \max_{i=1 \ldots n} p_i(.)
\end{equation}

\noindent where $w^*$ is the unknown minimizer defined by \eqref{eq1.0}, and $p_i(.): \Gamma_f \rightarrow \mathcal{R}$ are proximity control functions (or simply prox-functions) defined by $p_i(.)=f(x_i)+\langle \bigtriangledown f(x_i), .-x_i \rangle + \frac{\lambda}{2} ||.-x_i||^2$.

\end{lemma}

\begin{proof}

By strong convexity of $f(.)$ we have:

\begin{flalign}
	&	&\mathbb{B}_f(.||x_i) &\geq \frac{\lambda}{2} ||.-x_i||^2 && \nonumber \\
	&\Longrightarrow& f(.)	&\geq p_i(.) && \nonumber \\
	&\Longrightarrow& f(.)	&\geq \max_{i=1,\ldots,n} p_i(.) && \label{eq3.1.6c}\\
	&\Longrightarrow& \min f(.) = f(w^*) &\geq \max_{i=1,\ldots,n} p_i(w^*) \geq \min \max_{i=1,\ldots,n} p_i(.) && \label{eq3.1.6d}
\end{flalign}

\noindent where $\mathbb{B}_f(x_1||x_2)=f(x_1)-f(x_2)-\langle \bigtriangledown f(x_2),x_1-x_2 \rangle$ denotes the Bregman divergence between two points $x_1,x_2 \in \Gamma_f$. Both sides of \eqref{eq3.1.6c} and \eqref{eq3.1.6d} become equal if $f(.) = \max_{i=1,\ldots,n} p_i(.)$, so this bound cannot be improved without any extra condition.

\end{proof}

\begin{lemma}\label{state3.1.2}

(Jensen's inequality for strongly convex function) Given $n$ deterministic oracles $\Omega_n=\{ \omega(x_1), \ldots, \omega(x_n) \}$ defined by \eqref{eq3.2}. If $f(.)$ is a $\lambda$-strongly convex function, then for all $\alpha_1,\ldots,\alpha_n$ that satisfy $\sum_{i=1}^n \alpha_i = 1, \alpha_i \geq 0 \quad \forall i$, $f(y)$ is unimprovably upper bounded by:

\begin{equation}\label{eq3.2.5}
	f(y) \leq \sum_{i=1}^n \alpha_i f(x_i) - \frac{\lambda}{2} \sum_{i=1}^{n} \alpha_i ||x_i-y||^2
\end{equation}

\noindent where $y=\sum_{i=1}^{n} \alpha_i x_i$.

\end{lemma}

\begin{proof}

By strong convexity of $f(.)$ we have:

\begin{flalign*}
	&& \mathbb{B}_f(x_i||y) &\geq \frac{\lambda}{2} ||x_i-y||^2 &&\\
	&\Longrightarrow& f(y) &\leq f(x_i) -\langle \bigtriangledown f(y) , x_i-y \rangle - \frac{\lambda}{2} ||x_i-y||^2 &&\\
	&\Longrightarrow& f(y) &\leq \sum_{i=1}^n \alpha_i f(x_i) - \langle \bigtriangledown f(y), \sum_{i=1}^n \alpha_i x_i-y \rangle - \frac{\lambda}{2} \sum_{i=1}^{n} \alpha_i || x_i-y ||^2 &&\\
	&&	&\leq \sum_{i=1}^n \alpha_i f(x_i) - \frac{\lambda}{2} \sum_{i=1}^{n} \alpha_i ||x_i-y||^2 &&
\end{flalign*}

Both sides of all above inequalities become equal if $f(.)=\frac{\lambda}{2} ||.-y||^2 + \langle c_1,. \rangle + c_2$, where $c_1$ and $c_2$ are constants, so this bound cannot be improved without any extra condition.

\end{proof}

Immediately, the optimal $A$ that yields the lowest upper bound of $f(y)$ can be given by:

\begin{equation}\label{eq3.2.6}
	A = \arg\min_{\substack{\sum_{i=1}^n \alpha_i = 1 \\ \alpha_i \geq 0 \forall i}} \left\{ \sum_{i=1}^n \alpha_i f(x_i) - \frac{\lambda}{2} \sum_{i=1}^{n} \alpha_i ||x_i-\sum_{j=1}^{n} \alpha_j x_j||^2 \right\}
\end{equation}

Combining with \eqref{eq3.1}, \eqref{eq3.1.5}, we have an deterministic upper bound of $S(y)$:

\begin{equation}\label{eq3.2.7}
	S(y) \leq \min_{\substack{\sum_{i=1}^n \alpha_i = 1 \\ \alpha_i \geq 0 \forall i}} \left\{ \sum_{i=1}^n \alpha_i f(x_i) - \frac{\lambda}{2} \sum_{i=1}^{n} \alpha_i ||x_i-\sum_{j=1}^{n} \alpha_j x_j||^2 \right\} - \max_{i=1 \ldots n} p_i(w^*)
\end{equation}

This bound is quite useless at the moment as we are only interested in bounds in stochastic programming. The next lemma will show how it can be generalized in later case.

\begin{lemma}\label{state3.2}

Given $n$ stochastic oracles $\tilde{\Omega}_n=\{\tilde{\omega}(x_1),\ldots,\tilde{\omega}(x_n)\}$ defined by \eqref{eq1.1}, if $y(.,\ldots,.):  \mathcal{H}^n \times \Gamma_f^n \rightarrow \Gamma_f$ and $U(.,\ldots,.): \mathcal{H}^n \times \Gamma_f^n \rightarrow \mathcal{R}$ are functionals of $\tilde{f}_i(.)$ and $x_i$ that satisfy:

\begin{subequations}
\begin{align}
	&U(f,\ldots,f,x_1,\ldots,x_n) \geq S(y(f,\ldots,f,x_1,\ldots,x_n)) \label{eq3.3}\\
	&U(\tilde{f}_1,\ldots,\tilde{f}_n,x_1,\ldots,x_n) \text{ is convex w.r.t. } \tilde{f}_1,\ldots,\tilde{f}_n \label{eq3.4} \\
	&\mathbb{E}[\langle \bigtriangledown_{f,\ldots,f} U(f,\ldots,f,x_1,\ldots,x_n),[\tilde{f}_1-f,\ldots,\tilde{f}_n-f]^T \rangle] \leq 0 \label{eq3.4c}
\end{align}
\end{subequations}

\noindent then $\mathbb{E}[S(y(f,\ldots,f,x_1,\ldots,x_n))]$ is upper bounded by $U(\tilde{f}_1,\ldots,\tilde{f}_n,x_1,\ldots,x_n)$.

\end{lemma}

\begin{proof}

Assuming that $\delta_i(.): \Gamma_f \rightarrow \mathcal{R}$ are perturbation functions defined by

\begin{equation}\label{eq3.3.2}
	\delta_i(.)=\tilde{f}_i(.)-f(.)
\end{equation}

\noindent we have:

\begin{align}
	U(\tilde{f}_{1,\ldots,n},x_{1,\ldots,n}) &\geq U(f+\delta_1,\ldots,f + \delta_n,x_{1,\ldots,n}) \nonumber \\
	\text{(by \eqref{eq3.4})} &= U(f,\ldots,f,x_{1,\ldots,n}) + \langle \bigtriangledown_{f,\ldots,f} U(f,\ldots,f,x_{1,\ldots,n}),[\delta_{1,\ldots_n}]^T \rangle \nonumber \\
	\text{(by \eqref{eq3.3})} &\geq S(y(f,\ldots,f,x_{1,\ldots,n})) + \langle \bigtriangledown_{f,\ldots,f} U(f,\ldots,f,x_{1,\ldots,n}),[\delta_{1,\ldots,n}]^T \rangle \label{eq3.3.3}
\end{align}

Moving $\delta_i$ to the left side:

\begin{align*}
	\mathbb{E}[S(y(f,\ldots,f,x_{1,\ldots,n}))] &\leq U(\tilde{f}_{1,\ldots,n},x_{1,\ldots,n}) + \mathbb{E}[\langle \bigtriangledown_{f,\ldots,f} U(f,\ldots,f,x_{1,\ldots,n}),[\delta_{1,\ldots,n}]^T \rangle] \\
	\text{(by \eqref{eq3.4c})}	&\leq U(\tilde{f}_{1,\ldots,n},x_{1,\ldots,n})
\end{align*}

\end{proof}

Clearly, according to \eqref{eq3.4}, setting:

\begin{equation}\label{eq3.3.1}
	U(\tilde{f}_{1,\ldots,n},x_{1,\ldots,n}) = \min_{\substack{\sum_{i=1}^n \alpha_i = 1 \\ \alpha_i \geq 0 \forall i}} \left\{ \sum_{i=1}^n \alpha_i \tilde{f}_i(x_i) - \frac{\lambda}{2} \sum_{i=1}^{n} \alpha_i ||x_i-\sum_{j=1}^{n} \alpha_j x_j||^2 \right\} - \max_{i=1 \ldots n} \tilde{p}_i(w^*)
\end{equation}

\noindent by substituting $f(.)$ and $p_i(.)$ in \eqref{eq3.2.7} respectively with $\tilde{f}_i(.)$ defined by \eqref{eq1.2} and $\tilde{p}_i(.): \Gamma_f \rightarrow \mathcal{R}$ defined by:

\begin{equation}\label{eq3.3.5}
	\tilde{p}_i(.)=\tilde{f}_i(x_i)+\langle \bigtriangledown \tilde{f}_i(x_i), .-x_i \rangle+\frac{\lambda}{2} ||.-x_i||^2
\end{equation}

\noindent is not an option, because $\min_{\substack{\sum_{i=1}^n \alpha_i = 1 \\ \alpha_i \geq 0 \forall i}} \{.\}$ and $- \max_{i=1 \ldots n} \{.\}$ are both concave, $\sum_{i=1}^n \alpha_i \tilde{f}_i(x_i)$ and $\tilde{p}_i(w^*)$ are both linear to $\tilde{f}_i(.)$, and $\frac{\lambda}{2} \sum_{i=1}^{n} \alpha_i ||x_i-\sum_{j=1}^{n} \alpha_j x_j||^2$ is irrelevant to $\tilde{f}_i(.)$. This prevents asymptotically fast cutting-plane/bundle methods \cite{joachims2006cuttingplane,teo2009bundle,franc2009optimized} from being directly applied on stochastic oracles without any loss of performance. As a result, to decrease \eqref{eq3.1} and satisfy \eqref{eq3.4} our options boil down to replacing $\min_{\substack{\sum_{i=1}^n \alpha_i = 1 \\ \alpha_i \geq 0 \forall i}} \{.\}$ and $- \max_{i=1 \ldots n} \{.\}$ in \eqref{eq3.3.1} with their respective lowest convex upper bound:

\begin{align}
	U_{(A,B)} (\tilde{f}_{1,\ldots,n},x_{1,\ldots,n}) &= \sum_{i=1}^n \alpha_i \tilde{f}_i(x_i) - \frac{\lambda}{2} \sum_{i=1}^{n} \alpha_i ||x_i-\sum_{j=1}^{n} \alpha_j x_j||^2 - \sum_{i=1}^{n} \beta_i \tilde{p}_i(w^*) \nonumber \\
		&= \sum_{i=1}^n \alpha_i \tilde{f}_i(x_i) - \frac{\lambda}{2} \sum_{i=1}^{n} \alpha_i ||x_i-\sum_{j=1}^{n} \alpha_j x_j||^2 - \tilde{P}_n(w^*) \label{eq3.8}
\end{align}

\noindent where $\tilde{P}_n(.): \Gamma_f \rightarrow \mathcal{R}$ is defined by:

\begin{equation}\label{eq3.5.5}
	\tilde{P}_n(.) = \sum_{i=1}^{n} \beta_i \tilde{p}_i(.)
\end{equation}

\noindent and $A=[\alpha_1,\ldots,\alpha_n]^T$, $B=[\beta_1,\ldots,\beta_n]^T$ are constant $n$-dimensional vectors, with each $\alpha_i$, $\beta_i$ satisfying:

\begin{subequations}\label{eq3.6}
\begin{align}
	\sum_{i=1}^n \alpha_i &= 1 & \alpha_i &\geq 0 \quad \forall i\\
	\sum_{i=1}^n \beta_i &= 1 & \beta_i &\geq 0 \quad \forall i
\end{align}
\end{subequations}

\noindent accordingly $y((\tilde{f}_1,\ldots,\tilde{f}_n,x_1,\ldots,x_n)$ can be set to:

\begin{equation}\label{eq3.8.5}
	y_{(A,B)}(x_1,\ldots,x_n) = \sum_{i=1}^n \alpha_i x_i
\end{equation}

\noindent such that \eqref{eq3.3} is guaranteed by lemma \ref{state3.1.2}. It should be noted that $A$ and $B$ must both be constant vectors that are independent from all stochastic variables, otherwise the convexity condition \eqref{eq3.4} may be lost. For example, if we always set $\beta_i$ as the solution of the following problem:

\begin{equation*}
	\beta_i=\arg \max_{\substack{\sum_{i=1}^n \beta_i = 1 \\ \beta_i \geq 0 \forall i}} \left\{ \sum_{i=1}^{n} \beta_i \tilde{p}_i(w^*) \right\}
\end{equation*}

\noindent then $\tilde{P}_n(w^*)$ will be no different from the cutting-plane bound \eqref{eq3.1.5}. Finally, \eqref{eq3.4c} can be validated directly by substituting \eqref{eq3.8} back into \eqref{eq3.4c}:

\begin{align}\label{eq3.9}
	\langle \bigtriangledown_{f,\ldots,f} U_{(A,B)}(f,\ldots,f,x_{1,\ldots,n}),[\delta_{1,\ldots,n}]^T \rangle = \sum_{i=1}^{n} \left[ (\alpha_i-\beta_i) \delta_i(x_i) - \langle \bigtriangledown \delta_i(x_i), \beta_i (w^*-x_i) \rangle \right]
\end{align}

Clearly $\mathbb{E}[(\alpha_i-\beta_i) \delta_i(x_i)] = 0$ and $\mathbb{E}[\langle \bigtriangledown \delta_i(x_i), w^* \rangle] = 0$ can be easily satisfied because $\alpha_i$ and $\beta_i$ are already set to constants to enforce \eqref{eq3.4}, and by definition $w^*=\arg\min f(.)$ is a deterministic (yet unknown) variable in our problem setting, while both $\delta_i(x_i)$ and $\bigtriangledown \delta_i(x_i)$ are unbiased according to \eqref{eq1.2a}. Bounding $\mathbb{E}[\langle \bigtriangledown \delta_i(x_i), x_i \rangle]$ is a bit harder but still possible: In all optimization algorithms, each $x_i$ can either be a constant, or chosen from $\Gamma_f$ based on previous $\tilde{f}_1(.),\ldots,\tilde{f}_{i-1}(.)$ ($x_i$ cannot be based on $\tilde{f}_i(.)$ that is still unknown by the time $x_i$ is chosen). By the i.i.d. condition \eqref{eq1.2b}, they are all independent from $\tilde{f}_i(.)$, which implies that $x_i$ is also independent from $\tilde{f}_i(x_i)$:

\begin{equation}\label{eq3.9.1}
	\mathbb{E}[ \langle \bigtriangledown \delta_i(x_i), x_i \rangle ] = 0
\end{equation}

As a result, we conclude that \eqref{eq3.9} satisfies $\mathbb{E}[\langle \bigtriangledown_{f,\ldots,f} U_{(A,B)}(f,\ldots,f,x_{1,\ldots,n}),[\delta_{1,\ldots,n}]^T \rangle] = 0$, and subsequently $U_{(A,B)}$ defined by \eqref{eq3.8} satisfies all three conditions of Lemma \ref{state3.2}. At this point we may construct an algorithm that uniformly reduces $\max_{w^*} U_{(A,B)}$ by iteratively calling new stochastic oracles and updating $A$ and $B$. Our main result is summarized in the following theorem:

\begin{theorem}\label{state3.4}

For all $\lambda$-strongly convex function $F(.)$, assuming that at some stage of an algorithm, $n$ stochastic oracles $\tilde{\omega}(x_1),\ldots,\tilde{\omega}(x_n)$ have been called to yield a point $y_{(A^n,B^n)}$ defined by \eqref{eq3.8.5} and an upper bound $\hat{U}_{(A^n,B^n)}$ defined by:

\begin{align}
	\hat{U}_{(A^n,B^n)} (\tilde{f}_{1,\ldots,n},x_{1,\ldots,n}) &= U_{(A^n,B^n)} (\tilde{f}_{1,\ldots,n},x_{1,\ldots,n}) + \frac{\lambda}{2} \sum_{i=1}^{n} \alpha_i ||x_i-\sum_{j=1}^{n} \alpha_j x_j||^2 +\left( \tilde{P}_n(w^*) - \min \tilde{P}_n(.) \right) \nonumber \\
		&= \sum_{i=1}^n \alpha_i \tilde{f}_i(x_i) - \min \tilde{P}_n(.) \label{eq3.5}
\end{align}

\noindent where $A^n$ and $B^n$ are constant vectors satisfying \eqref{eq3.6} (here $^n$ in $A^n$ and $B^n$ denote superscripts, should not be confused with exponential index), if the algorithm calls another stochastic oracle $\tilde{\omega}(x_{n+1})$ at a new point $x_{n+1}$ given by:

\begin{align*}
	x_{n+1} &=\arg\min \tilde{P}_n(.)
\end{align*}

\noindent and update $A$ and $B$ by:

\begin{subequations}\label{eq3.10}
\begin{align}
	A^{n+1} &=\left[ \left( 1-\frac{\lambda}{G^2} \hat{U}_{(A^n,B^n)} \right) (A^n)^T , \frac{\lambda}{G^2} \hat{U}_{(A^n,B^n)} \right]^T \\
	B^{n+1} &=\left[ \left( 1-\frac{\lambda}{G^2} \hat{U}_{(A^n,B^n)} \right) (B^n)^T , \frac{\lambda}{G^2} \hat{U}_{(A^n,B^n)} \right]^T
\end{align}
\end{subequations}
 
\noindent where $G=\max ||\bigtriangledown \tilde{f}_i(.)||$, then $\hat{U}_{A^{n+1},B^{n+1}}$ is bounded by:

\begin{equation}\label{eq3.11}
	\hat{U}_{(A^{n+1},B^{n+1})} \leq \hat{U}_{(A^n,B^n)}-\frac{\lambda}{2 G^2} \hat{U}_{(A^n,B^n)}^2
\end{equation}

\end{theorem}

\begin{proof}

First, optimizing and caching all elements of $A^n$ or $B^n$ takes at least $O(n)$ time and space, which is not possible in large scale problems. So we confine our options of $A^{n+1}$ and $B^{n+1}$ to setting:

\begin{subequations}\label{eq3.11.5}
\begin{align}
	[\alpha^{n+1}_1,\ldots,\alpha^{n+1}_n]^T &=(1-\alpha^{n+1}_{n+1}) [\alpha^n_1,\ldots,\alpha^n_n]^T\\
	[\beta^{n+1}_1,\ldots,\beta^{n+1}_n]^T &=(1-\beta^{n+1}_{n+1}) [\beta^n_1,\ldots,\beta^n_n]^T
\end{align}
\end{subequations}

\noindent such that previous $\sum_{i=1}^n \alpha_i \tilde{F}(x_i)$ and $\sum_{i=1}^n \beta_i \tilde{p_i}(.)$ can be summed up in previous iterations in order to produce a $1$-memory algorithm instead of an $\infty$-memory one, without violating \eqref{eq3.6}. Consequently $\hat{U}_{(A^{n+1},B^{n+1})}$ can be decomposed into:

\begin{align*}
	\hat{U}_{(A^{n+1},B^{n+1})} &= \sum_{i=1}^{n+1} \alpha^{n+1}_i \tilde{f}(x_i) - \min \tilde{P}_{n+1}(.) \\
	\text{(by \eqref{eq3.3.5}, \eqref{eq3.5.5}, \eqref{eq3.6})}	&\leq \sum_{i=1}^{n+1} \alpha^{n+1}_i \tilde{f}(x_i) - \left[ \sum_{i=1}^{n+1} \beta^{n+1}_i \tilde{p}_i(\tilde{x}^*_n) - \frac{1}{2 \lambda} ||\bigtriangledown \sum_{i=1}^{n+1} \beta^{n+1}_i \tilde{p}_i(\tilde{x}^*_n)||^2 \right] \\
	\text{(by \eqref{eq3.11.5})}	&= \left[ (1-\alpha^{n+1}_{n+1}) \sum_{i=1}^{n} \alpha^{n}_i \tilde{f}(x_i) - (1-\beta^{n+1}_{n+1}) \sum_{i=1}^{n} \beta^{n}_i \tilde{p}_i(\tilde{x}^*_n) \right]\\
		&\quad + \left[ \alpha^{n+1}_{n+1} \tilde{f}(x_{n+1}) - \beta^{n+1}_{n+1} \tilde{p}_{n+1}(\tilde{x}^*_n) \right] + \frac{(\beta^{n+1}_{n+1})^2}{2 \lambda} ||\bigtriangledown \tilde{p}_{n+1}(\tilde{x}^*_n)||^2
\end{align*}

\noindent where $\tilde{x}^*_n=\arg\min \tilde{P}_n(.)$, setting $a^{n+1}_{n+1}=b^{n+1}_{n+1}$ and $x_{n+1}=\tilde{x}^*_n$ eliminates the second term:

\begin{align}
	\hat{U}_{(A^{n+1},B^{n+1})} &= (1-\alpha^{n+1}_{n+1}) \left( \sum_{i=1}^{n} \alpha^{n}_i \tilde{f}(x_i) - \tilde{P}_n(\tilde{x}^*_n) \right) + \frac{(\alpha^{n+1}_{n+1})^2}{2 \lambda} ||\bigtriangledown \tilde{p}_{n+1}(x_{n+1})||^2 \nonumber \\
	\text{(by \eqref{eq3.3.5})}	&= (1-\alpha^{n+1}_{n+1}) \hat{U}_{(A^n,B^n)} + \frac{(\alpha^{n+1}_{n+1})^2}{2 \lambda} ||\bigtriangledown \tilde{f}_{n+1}(x_{n+1})||^2 \nonumber \\
	\text{($G \geq ||\bigtriangledown \tilde{f}_i(.)||$)} &\leq (1-\alpha^{n+1}_{n+1}) \hat{U}_{(A^n,B^n)} + \frac{(\alpha^{n+1}_{n+1})^2 G^2}{2 \lambda} \label{eq3.12}
\end{align}

Let $u_i=\frac{2 \lambda}{G^2} \hat{U}_{(A^i,B^i)}$, minimizing the right side of \eqref{eq3.12} over $\alpha^{n+1}_{n+1}$ yields:

\begin{equation}\label{eq3.12.5}
	\alpha^{n+1}_{n+1} = \arg\min_{\alpha} \{ (1-\alpha) u_n + \alpha^2 \} = \frac{u_n}{2} =\frac{\lambda}{G^2} \hat{U}_{(A^n,B^n)}
\end{equation}

In this case:

\begin{flalign}\label{eq3.13}
	&& u_{n+1} &\leq u_n-\frac{u_n^2}{4} &&\\
	&\Longrightarrow& \hat{U}_{(A^{n+1},B^{n+1})} &\leq \hat{U}_{(A^n,B^n)}-\frac{2 \lambda}{G^2} \hat{U}^2_{(A^n,B^n)} && \nonumber
\end{flalign}

\end{proof}

Given an arbitrary initial oracle $\tilde{\omega}(x_1)$ and apply the updating rule in theorem \ref{state3.4} recursively results in algorithm 1, accordingly we can prove its asymptotic behavior by induction:

\begin{corollary}\label{state3.5}

The final point $y_n$ obtained by applying algorithm 1 on arbitrary $\lambda$-strongly convex function $f(.)$ has the following worst-case rate of convergence:

\begin{equation*}
	\mathbb{E}[f(y_n)]-\min f(.) \leq \frac{2 G^2}{\lambda (n+3)}
\end{equation*}

\end{corollary}

\begin{proof}

First, by \eqref{eq3.13} we have:

\begin{equation}\label{eq3.14}
	\frac{1}{u_{n+1}} \geq \frac{1}{u_n \left( 1-\frac{u_n}{4} \right) } = \frac{1}{u_n} + \frac{1}{4-u_n} \geq \frac{1}{u_n} + \frac{1}{4}
\end{equation}

On the other hand, by strong convexity, for all $x_1 \in \Gamma_f$ we have:

\begin{equation}\label{eq3.15}
	f(x_1)-\min f(.) \leq \frac{1}{2 \lambda} ||\bigtriangledown f(x_1)||^2 \leq \frac{G^2}{2 \lambda}
\end{equation}

Setting $\hat{U}_{(1,1)}=\frac{G^2}{2 \lambda}$ as intial condition and apply \eqref{eq3.14} recursively induces the following generative function:

\begin{flalign*}
	&& \frac{1}{u_n} &\geq 1+\frac{n-1}{4}=\frac{n+3}{4} && \\
	&\Longrightarrow& u_n &\leq \frac{4}{n+3} &&\\
	&\Longrightarrow& \hat{U}_{(A^n,B^n)} &\leq \frac{2 G^2}{\lambda(n+3)}\\
	&\Longrightarrow& \mathbb{E}[f(y_n)]-\min f(.) &\leq \frac{2 G^2}{\lambda(n+3)} - \frac{\lambda}{2} \sum_{i=1}^{n} \alpha^{n}_i ||x_i-y_n||^2 - \left( \tilde{P}_n(w^*) - \min \tilde{P}_n(.) \right)
\end{flalign*}

\end{proof}

This worst-case rate of convergence is four times faster than Epoch-GD ($\frac{8 G^2}{\lambda n}$) \cite{hazan2011epochGD,Juditski2011uniform} or Cutting-plane/Bundle Method ($\frac{8 G^2}{ \lambda \left[ n+2-\log_{2} \left( \frac{\lambda f(x_1)}{4 G^2} \right) \right] }$) \cite{joachims2006cuttingplane,teo2009bundle,franc2009optimized}, and is indefinitely faster than SGD ($\frac{ln(n) G^2}{2 \lambda n}$) \cite{bottou2008SGD,shalev2007pegasos}.

\section{High Probability Bound}

An immediate result of Corollary \ref{state3.5} is the following high probability bound yielded by Markov inequality:

\begin{equation}\label{eq5.0}
	\mathrm{Pr} \left( S(y_n) \geq \frac{2 G^2}{\lambda (n+3) \eta} \right) \leq \eta
\end{equation}

where $1-\eta \in [0,1]$ denotes the confidence of the result $y_n$ to reach the desired suboptimality. In most cases (particularly when $\eta \approx 0$, as demanded by most applications) this bound is very loose and cannot demonstrate the true performance of the proposed algorithm. In this section we derive several high probability bounds that are much less sensitive to small $\eta$ comparing to \eqref{eq5.0}. 

\begin{corollary}

The final point $y_n$ obtained by applying algorithm 1 on arbitrary $\lambda$-strongly convex function $F(.)$ has the following high probability bounds:

\begin{subequations}\label{eq5.0.1}
\begin{align}
	\mathrm{Pr} \left( S(y_n) \geq t+\frac{2 G^2}{\lambda (n+3)} \right) &\leq \exp \left\{ -\frac{t^2 (n+2)}{16 D^2 \sigma^2} \right\} \\
	\mathrm{Pr} \left( S(y_n) \geq t+\frac{2 G^2}{\lambda (n+3)} \right) &\leq \frac{1}{2} \exp \left\{ - \frac{t^2 (n+2)}{8 \tilde{G}^2 D^2} \right\} \label{eq5.0.1b}\\
	\mathrm{Pr} \left( S(y_n) \geq t+\frac{2 G^2}{\lambda (n+3)} \right) &\leq \exp \left\{ -\frac{t (n+2)}{4 \tilde{G} D} \mathrm{ln} \left( 1+ \frac{t \tilde{G}}{2 D \sigma^2} \right) \right\} \label{eq5.0.1c}
\end{align}
\end{subequations}

where constants $\tilde{G}=\max ||\bigtriangledown \delta_i(.)||$, $\sigma^2=\max \mathrm{Var} ( \bigtriangledown \tilde{f}_i(.) )$ are maximal range and variance of each stochastic subgradient respectively, and $D=\max_{x_1,x_2 \in \Gamma_f} ||x_1-x_2||$ is the largest distance between two points in $\Gamma_f$.

\end{corollary}

\begin{proof}

We start by expanding the right side of \eqref{eq3.3.3}, setting $A^n=B^n$ and substituting \eqref{eq3.9} back into \eqref{eq3.3.3} yields:

\begin{align}
	S(y_n) &\leq U_{(A^n,A^n)}(\tilde{f}_{1,\ldots,n},x_{1,\ldots,n}) - \sum_{i=1}^{n} \alpha^n_i \langle \bigtriangledown \delta_i(x_i), x_i-w^* \rangle \nonumber \\
	\text{(by Corollary \ref{state3.4})}	&\leq \frac{2 G^2}{\lambda (n+3)} + \sum_{i=1}^{n} \alpha^n_i r_i \label{eq5.1}
\end{align}

\noindent with each $r_i = -\langle \bigtriangledown \delta_i(x_i), x_i-w^* \rangle$ satisfying:

\begin{align}
	\text{(Cauchy's inequality) } -||\bigtriangledown \delta_i(x_i)|| ||x_i-w^*|| &\leq r_i \leq ||\bigtriangledown \delta_i(x_i)|| ||x_i-w^*|| \nonumber \\
		-\tilde{G} D &\leq r_i \leq \tilde{G} D \label{eq5.2}
\end{align}

\begin{align}
	\mathrm{Var} ( r_i ) &= \mathbb{E}[( \langle \bigtriangledown \delta_i(x_i),x_i-w^* \rangle -\mathbb{E}[\langle \bigtriangledown \delta_i(x_i),x_i-w^* \rangle] )^2] \nonumber \\
	\text{(by \eqref{eq3.9.1})}	&= \mathbb{E}[( \langle \bigtriangledown \delta_i(x_i),x_i-w^* \rangle )^2] \nonumber \\
	\text{(Cauchy's inequality)}	&\leq \mathbb{E}[||\bigtriangledown \delta_i(x_i)||^2 ||x_i-w^*||^2] \nonumber \\
		&\leq D^2 \mathbb{E}[||\bigtriangledown \delta_i(x_i)||^2] = D^2 \mathrm{Var} \left( \bigtriangledown \tilde{f}_i(x_i) \right) \leq D^2 \sigma^2 \label{eq5.3}
\end{align}

% \begin{align*}
% 	\mathrm{Var} \left( \sum_{i=1}^{n} \alpha^n_i \langle \bigtriangledown \delta_i(x_i), x_i-x^*_n \rangle \right) &= \sum_{i=1}^n (\alpha^n_i)^2 \mathrm{Var} ( \langle \bigtriangledown \delta_i(x_i), x_i-x^*_n \rangle )\\
% 		&\leq D^2 \sigma^2 \sum_{i=1}^n (\alpha^n_i)^2\\
% \end{align*}

This immediately expose $S_n(y_{(A^n,A^n)}(x_{1,\ldots,n}))$ to several inequalities in non-parametric statistics that bound the probability of sum of independent random variables:

\begin{subequations}\label{eq5.4}
\begin{align}
	\text{(generalized Chernoff bound) }	\mathrm{Pr} \left( \sum_{i=1}^{n} \alpha^n_i r_i \geq t \right) &\leq \exp \left\{ -\frac{t^2}{4 \mathrm{Var} \left( \sum_{i=1}^{n} \alpha^n_i r_i \right) } \right\} \nonumber \\
	\text{(by \eqref{eq5.3})}	&\leq \exp \left\{ -\frac{t^2}{4 D^2 \sigma^2 \sum_{i=1}^{n} (\alpha^n_i)^2} \right\} \\
	\text{(Azuma-Hoeffding inequality) }	\mathrm{Pr} \left( \sum_{i=1}^{n} \alpha^n_i r_i \geq t \right) &\leq \frac{1}{2} \exp \left\{ - \frac{2 t^2}{\sum_{i=1}^{n} (\alpha^n_i)^2 (\max r_i - \min r_i)^2 } \right\} \nonumber \\
	\text{(by \eqref{eq5.2})}	&\leq \frac{1}{2} \exp \left\{ - \frac{2 t^2}{4 \tilde{G}^2 D^2 \sum_{i=1}^{n} (\alpha^n_i)^2} \right\} \\
	\text{(Bennett inequality) }	\mathrm{Pr} \left( \sum_{i=1}^{n} \alpha^n_i r_i \geq t \right) &\leq \exp \left\{ -\frac{t}{2 \max ||\alpha^n_i \epsilon_i||} \mathrm{ln} \left( 1+ \frac{t \max ||\alpha_i \epsilon_i||}{\mathrm{Var} \left( \sum_{i=1}^{n} \alpha^n_i r_i \right) } \right) \right\} \nonumber \\
	\text{(by \eqref{eq5.2}, \eqref{eq5.3})}	&\leq \exp \left\{ -\frac{t}{2 \tilde{G} D \max \alpha_i} \mathrm{ln} \left( 1+ \frac{t \tilde{G} \max \alpha^n_i}{D \sigma^2 \sum_{i=1}^{n} (\alpha^n_i)^2} \right) \right\}
\end{align}
\end{subequations}

In case of algorithm 1, if $A^n$ is recursively updated by \eqref{eq3.10}, then each two consecutive $\alpha^n_i$ has the following property:

\begin{flalign}
	&&	\text{(by \eqref{eq3.10}) } \frac{\alpha^n_{i+1}}{\alpha^n_i} &= \frac{\alpha^{i+1}_{i+1}}{\alpha^{i+1}_i} = \frac{\alpha^{i+1}_{i+1}}{\alpha^i_i (1-\alpha^{i+1}_{i+1})} && \nonumber \\
	&&	\text{(by \eqref{eq3.12.5}, \eqref{eq3.13})} &=\frac{u_{i-1}-\frac{u_{i-1}^2}{4}}{u_{i-1} \left( 1-\frac{u_{i-1}}{2}+\frac{u_{i-1}^2}{8} \right)} && \nonumber \\
	&&	\text{($u_{i-1} \leq 1$)} &> 1 && \nonumber \\
	&\Longrightarrow&	\alpha^n_{i+1} &> \alpha^n_{i} && \nonumber\\
	&\Longrightarrow&	\max \alpha^n_i &= \alpha^n_n = \frac{u_{n-1}}{2} \leq \frac{2}{n+2} && \label{eq5.5}
\end{flalign}

Accordingly $\sum_{i=1}^{n} (\alpha^n_i)^2$ can be bounded by

\begin{equation}\label{eq5.6}
	\sum_{i=1}^{n} (\alpha^n_i)^2 \leq n (\alpha^n_n)^2 \leq \frac{4 n}{(n+2)^2} \leq \frac{4}{n+2}
\end{equation}

Eventually, combining \eqref{eq5.1} \eqref{eq5.4}, \eqref{eq5.5} and \eqref{eq5.6} together yields the proposed high probability bounds \eqref{eq5.0.1}.

\end{proof}

By definition $\tilde{G}$ and $\sigma$ are both upper bounded by $G$. And if $\Gamma_f$ is undefined, by combining strong convexity condition $\mathbb{B}_{f} (x_1||\arg\min f(.)) = f(x_1)-\min f(.) \geq \frac{\lambda}{2} ||x_1-\arg\min f(.)||^2$ and \eqref{eq3.15} together we can still set

\begin{equation*}
	\Gamma_f=\left\{ ||. - x_1||^2 \leq \frac{G^2}{\lambda^2} \right\}
\end{equation*}

\noindent such that $D = \frac{2 G}{\lambda}$, while $\arg\min f(.)$ is always included in $\Gamma_f$. Consequently, even in worst cases \eqref{eq5.0} can be easily superseded by any of \eqref{eq5.0.1}, in which $\eta$ decreases exponentially with $t$ instead of inverse proportionally. In most applications both $\tilde{G}$ and $\sigma$ can be much smaller than $G$, and $\sigma$ can be further reduced if each $\tilde{\omega}(x_i)$ is estimated from averaging over several stochastic oracles provided simultaneously by a parallel/distributed system.

\section{Discussion}

In this article we proposed algorithm 1, a first-order algorithm for stochastic strongly convex optimization that asymptotically outperforms all state-of-the-art algorithms by four times, achieving less than $S$ suboptimality using only $\frac{2 G^2}{\lambda S}-3$ iterations and stochastic oracles in average. Theoretically algorithm 1 can be generalized to strongly convex functions w.r.t. arbitrary norms using technique proposed in \cite{Juditski2011uniform}, and a slightly different analysis can be used to find optimal methods for strongly smooth (a.k.a. gradient lipschitz continuous or g.l.c.) functions, but we will leave them to further investigations. We do not know if this algorithm is optimal and unimprovable, nor do we know if higher-order algorithms can be discovered using similar analysis. There are several loose ends we may possibly fail to scrutinize, clearly, the most likely one is that we assume:

\begin{equation*}
	\max_f S(y) = \max_f \{ f(y)-\min f(.) \} \leq \max_f f(y) - \min_f \min f(.)
\end{equation*}

However in fact, there is no case $\arg\max_f f(y) = \arg\min_f \min f(.) \quad \forall y \in \Gamma_f$, so this bound is still far from unimprovable. Another possible one is that we do not know how to bound $\frac{\lambda}{2} \sum_{i=1}^{n} \alpha^n_i ||x_i-y_n||^2$ by optimizing $x_n$ and $\alpha_{n}^{n}$, so it is isolated from \eqref{eq3.5} and never participate in parameter optimization of \eqref{eq3.12}, but actually it can be decomposed into:

\begin{align*}
	\sum_{i=1}^{n+1} \alpha^{n+1}_i ||x_i-y_{n+1}||^2 &= \min \sum_{i=1}^{n+1} \alpha^{n+1}_i ||x_i-.||^2 \\
		&= \sum_{i=1}^{n+1} \alpha^{n+1}_i ||x_i-y_n||^2 - \frac{1}{2 \lambda} || \bigtriangledown_{y_n} \left\{ \sum_{i=1}^{n+1} \alpha_i ||x_i-y_n||^2 \right\} ||^2 \\
		&= \sum_{i=1}^{n} \alpha^{n+1}_i ||x_i-y_n||^2 + \alpha_{n+1} \left[ ||x_{n+1}-y_n||^2 \right] - \frac{(\alpha^{n+1}_{n+1})^2}{2 \lambda} || y_n-x_{n+1} ||^2 \\
		&= (1-\alpha^{n+1}_{n+1}) \left[ \sum_{i=1}^{n} \alpha^n_i ||x_i-y_n||^2 \right] + \left[ \alpha_{n+1}-\frac{(\alpha^{n+1}_{n+1})^2}{2 \lambda} \right] || y_n-x_{n+1} ||^2
\end{align*}

\noindent such that $\left[ \alpha_{n+1}-\frac{(\alpha^{n+1}_{n+1})^2}{2 \lambda} \right] || y_n-x_{n+1} ||^2$ can be added into the right side of \eqref{eq3.12}, unfortunately, we still do not know how to bound it, but it may be proved to be useful in some alternative problem settings (e.g. in optimization of strongly smooth functions).

Most important, if $f(.)$ is $\lambda$-strongly convex and each $\tilde{f}_i(.)$ can be revealed completely by each oracle (instead of only its first-order information), then the principle of empirical risk minimization (ERM):

\begin{equation*}
	y_n = \arg\min \sum_{i=1}^n \tilde{f}_i(.)
\end{equation*}

\noindent easily outperforms all state-of-the-art stochastic methods by yielding the best-ever rate of convergence $\frac{\sigma^2}{2 \lambda n}$ \cite{sridharan2008fast}, and is still more than four times faster than algorithm 1 (through this is already very close for a first-order method). This immediately raises the question: how do we close this gap? and if first-order methods are not able to do so, how much extra information of each $\tilde{f}_i(.)$ is required to reduce it? We believe that solutions to these long term problems are vital in construction of very large scale predictors in computational learning, but we are still far from getting any of them.

% \subsection{Does Proximal Regularization Technique Help?}
% 
% proximal regularization technique was popularized in convex and stochastic programming research by several seminar publications \cite{Juditski2011uniform} that claimed significant performance improvements. Basically, this technique can be introduced on arbitrary first-order method by setting up a proximal regularization function $R(.)$ such that:
% 
% \begin{equation}
% 	\min f(.) \leq \min { f(.)+R(.) } - \max_{\Gamma_f} R(.)
% \end{equation}

\section*{Acknowledgements}

We want to thank Dr Francesco Orabona and Professor Nathan Srebro for their critical comments on this manuscript. We also want to thank Maren Sievers, Stefan Rothschuh, Raouf Rusmaully and Li He for their proof reading and/or technical comments on an C++ implementation of this method.

\bibliographystyle{plain}
\bibliography{reference}

\end{document}